\documentclass[12pt]{article}

\usepackage{geometry} 

\usepackage{appendix}

\usepackage{cite}
\usepackage{amsmath,amssymb,amsfonts}
\usepackage{algorithmic}
\usepackage{graphicx}
\usepackage{algorithm,algorithmic}
\usepackage{textcomp}

\usepackage{graphicx} 
\usepackage{amsmath,amssymb,amsfonts}
\usepackage{algorithmic}
\usepackage{algorithm}
\usepackage{dsfont}
\usepackage{graphbox}
\usepackage{textcomp}
\usepackage{multirow}
\usepackage{balance}
\usepackage[export]{adjustbox}
\usepackage{color}
\usepackage{accents}
\usepackage{textcomp}
\usepackage{comment}
\usepackage{bm}
\usepackage{upgreek}
\usepackage{amsthm}
\newtheorem{theorem}{Theorem}
\newtheorem{corollary}{Corollary}
\newtheorem{lemma}{Lemma}
\newtheorem{remark}{Remark}
\newtheorem{definition}{Definition}
\newtheorem{proposition}{Proposition}
\newtheorem{assumption}{Assumption}

\newtheorem{example}{Example}

\newcommand{\mfs}[1]{{\normalfont\textsf{#1}}}
\newcommand{\mf}{\mathbf}

\begin{document}
\title{\Huge Towards Parameter-free Distributed Optimization: a Port-Hamiltonian Approach}
\author{Rodrigo Aldana-López, Alessandro Macchelli, Giuseppe Notarstefano,\\ Rosario Aragüés, Carlos Sagüés
\footnote{This work was supported via projects PID2021-124137OB-I00 and TED2021-130224B-I00
funded by MCIN/AEI/10.13039/501100011033, by ERDF A way of making Europe and by the
European Union NextGenerationEU/PRTR, by the Gobierno de Aragón under Project DGA T45-23R, by the Universidad de Zaragoza and Banco Santander, by the Consejo Nacional de Ciencia y Tecnología (CONACYT-Mexico) with grant number 739841. \\R. Aldana-López, R. Aragüés and C. Sagüés are with Departamento de Informática e Ingeniería de Sistemas (DIIS) and Instituto de Investigación en Ingeniería de Aragón (I3A), 
Universidad de Zaragoza, Zaragoza 50018, Spain.
(e-mail:  {rodrigo.aldana.lopez@gmail.com, \{raragues, csagues\}@unizar.es}) \\ A. Macchelli and G. Notarstefano are with the Department of Electrical, Electronic and Information Engineering (DEI), University of Bologna, 40136 Bologna, Italy (e-mail: \{alessandro.macchelli, giuseppe.notarstefano\}@unibo.it).}}

\maketitle

\begin{abstract}
This paper introduces a novel distributed optimization technique for networked systems, which removes the dependency on specific parameter choices, notably the learning rate. Traditional parameter selection strategies in distributed optimization often lead to conservative performance, characterized by slow convergence or even divergence if parameters are not properly chosen. In this work, we propose a systems theory tool based on the port-Hamiltonian formalism to design algorithms for consensus optimization programs. Moreover, we propose the Mixed Implicit Discretization (MID), which transforms the continuous-time port-Hamiltonian system into a discrete time one, maintaining the same convergence properties regardless of the step size parameter. The consensus optimization algorithm enhances the convergence speed without worrying about the relationship between parameters and stability. Numerical experiments demonstrate the method's superior performance in convergence speed, outperforming other methods, especially in scenarios where conventional methods fail due to step size parameter limitations.
\end{abstract}

\section{Introduction}
\label{sec:introduction}

Current technological trends lean towards the extended use of machine learning algorithms with increasing complexity and computing power requirements. Fortunately, computation and communication capabilities are readily available in modern devices. For this reason, distributed implementations have emerged as scalable and robust alternatives to deploy these technologies. Since many tasks of importance, e.g., the training of neural networks can be posed as optimization programs, distributed optimization has become a persistently active area of research in the last years (see \cite{Notarstefano2019, yang2019}). The problem of consensus optimization is particularly important, where the unconstrained cost is comprised of the sum of local costs that can be accessed at each device in the network separately \cite{Notarstefano2019}. 

In the following, we review the literature on algorithms developed for distributed optimization problems. First, we examine algorithms formulated in continuous time, leveraging tools from systems theory to guarantee stability. Subsequently, we discuss methodologies that operate in discrete time, reflecting more practical communication scenarios. Finally, we mention works using the Port-Hamiltonian System (PHS) formalism in multi-agent contexts and parameter-free discretization.

Many algorithms for addressing distributed optimization problems through continuous-time systems exist in the literature. For instance, a method for consensus optimization over weighted digraphs was introduced in \cite{gharesifard2014distributed}. A different approach, employing a proportional-integral protocol to tackle constrained distributed optimization, was proposed in \cite{yang2017multi}. The strategy in \cite{zeng2017distributed} aims to resolve constrained distributed convex optimization problems via non-smooth analysis. Meanwhile, \cite{lin2017distributed} presented a protocol with finite-time convergence in solving constrained distributed optimization issues. The use of adaptive control techniques for unconstrained distributed optimization was explored in \cite{li2018distributed}. A passivity-based methodology for designing and analyzing distributed protocol algorithms in the unconstrained scenario was used in \cite{Hatanaka2018}, with a similar approach adopted for balanced digraphs in \cite{li2020input}.
Furthermore, a distributed protocol leveraging Newton-Raphson iterations was detailed in \cite{moradian2022distributed}. These contributions underscore the benefits of modeling distributed optimization algorithms in continuous time, allowing for the application of systems theory tools, such as Lyapunov theory or passivity. However, given that communication among agents in digital platforms occurs in discrete time, research has pivoted towards considering discrete time interactions. A leading methodology for consensus optimization in discrete time, known as the gradient tracking method, was introduced in \cite{nedic2017achieving}, incorporating a dynamic average consensus component to track the gradient of the objective function. This method has spawned various adaptations, including those tailored for nonconvex programs \cite{dilorenzo2016next} and protocols integrating Newton-Raphson iterations \cite{varagnolo2016newton}. Extensions to accommodate time-varying and stochastic networks were proposed in \cite{nedic2017achieving, xu2017convergence}, respectively. Efforts to enhance convergence speed are used in \cite{xi2017addopt,qu2018harnessing}. Techniques such as push-sum and push-pull protocols for message passing, aimed at addressing digraphs, were applied in \cite{scutari2019distributed, pu2020push}. An adaptive momentum term to boost convergence rates was introduced in \cite{carnevale2020distributed}. Using a systems structure perspective, most of these variations can be understood using a general structural framework as in \cite{vanscoy2022universal}. Despite their generally effective performance, these algorithms constrain specific parameters, like the learning rate. Tuning such parameters is critical, especially in multi-agent settings where the constraints depend on global information such as the network topology. For instance, a large learning rate may result in instability, and a small learning rate may lead to slow convergence. 

An alternative line of research for incorporating discrete time communication involves employing continuous-time systems with event-triggered communication. One of the first works in this context is \cite{kia2015distributed}. Similar settings, including quantized information, were explored in \cite{liu2016event}. Extensions for quadratic optimization programs with constraints were considered in \cite{zhao2018distributed}. Event-triggered mechanisms for second-order dynamics agents were investigated in \cite{yi2018distributed}. The works  \cite{kajiyama2018distributed, adachi2021distributed} explored edge-based approaches for the design of the triggering conditions. A noteworthy contribution in this domain is \cite{guido2023}, which extends the continuous-time version of the gradient tracking algorithm to include event-triggered communication. While adopting many advantages from continuous-time methodologies, these approaches also need an upper limit on the interval between communications, as indicated in \cite{kia2015distributed}. Additionally, they require the implementation of a discrete time variant for the continuous-time algorithm, exemplified by \cite{guido2023}, where the event-triggering condition is evaluated at each timestep during experiments. Hence, parameter constraints cannot be avoided either in this case.

Henceforth, this motivates the need to develop distributed optimization methods that are not limited by design parameter constraints. An important breakthrough for parameter-free optimization was presented in \cite{portHgrad} where the Port-Hamiltonian System (PHS) formalism was used to design a continuous-time gradient flow for unconstrained optimization in a single-agent setting. This formalism allowed the authors to propose a parameter-free symplectic discretization. Generally, a PHS is a class of nonlinear systems describing a wide range of systems of interest, particularly suitable for power-conserving interconnections between subsystems \cite{Schaft2017}. This formalism has been proved useful to describe and analyze multi-agent algorithms for consensus \cite{Jieqiang2017consensus} and distributed control in \cite{van2010}. Moreover, geometric integration using the discrete gradient has been shown to preserve the PHS structure \cite{Macchelli2023,Macchelli2022}, with a precise parallelism between continuous and discrete time. In addition, under suitable conditions, this type of discretization can ensure convergence regardless of the step size parameter. For instance, the discrete gradient method was used in \cite{discretegrad} to develop a parameter-free optimization algorithm. However, despite the advantages of PHS and geometric integration methods, these tools have yet to be used to solve distributed optimization problems. 

The contributions of this work are twofold. First, we propose a PHS framework for the first time to desig and analyze distributed optimization algorithms. In particular, we apply it to solve an unconstrained consensus optimization program, where the convergence proofs are based on PHS tools. Second, our proposal allows the design in continuous time with a novel parameter-free discretization, referred to here as Mixed Implicit Discretization (MID), that enables the preservation of the PHS structure. In this sense, all agents in the network choose a step size parameter without worrying about the instability of the overall system. In addition, we discuss how having the step size as a free design parameter can be further used to accelerate convergence.

This manuscript is organized as follows. First, in Section \ref{sec:continous:time}, we provide a problem statement and some preliminaries on PHS. Section \ref{sec:phsopt} presents our PHS-based framework for consensus optimization in continuous time, including some concrete examples. Section \ref{sec:discrete:time} discusses parameter-free discretization. In Section \ref{sec:simulations}, we provide several numerical examples, comparing the proposal's performance against other methods in the literature. Conclusions are drawn in Section \ref{sec:conclusions}, and auxiliary results and proofs are placed in the appendices.

\textbf{Notation:} For a set $S$, let $|S|$ represent its cardinality. Given vectors $\{\mf{x}_i\}_{i\in\mathcal{I}}$ with index set $\mathcal{I}=\{1,\dots,\mfs{N}\}$, we denote with $[\mf{x}_i]_{i\in\mathcal{I}}=[\mf{x}_1^\top,\dots,\mf{x}_\mfs{N}^\top]^\top$. Denote with $\nabla$ the gradient operator. Let $\mathds{1}_n, \mathbf{0}_{n\times m}, \mathbf{I}_n$ denote the column vector of $n$ ones, the zero $n\times m$ matrix and the $n\times n$ identity matrix respectively. We write $\mathds{1},\mf{0},\mf{I}$ when there is no ambiguity on the dimensions, according to the context. The symbol $\otimes$ represents the Kronecker product. For discrete time sequences $\{\mf{x}[k]\}_{k\in\mathbb{N}}$, we often write $\mf{x}$ represents the value of the sequence at an arbitrary time step, $\mf{x}^+$ represents its value at the next time step. Let $\|\mf{x}\|$ represent the Euclidean norm of a vector $\mf{x}$ and $\|\mf{X}\|$ represent the induced Euclidean norm of a matrix $\mf{X}$.

\section{Poblem statement and preliminaries}
\label{sec:continous:time}
\subsection{Problem statement}
\label{sec:problem}
Consider a group of $\mfs{N}\in\mathbb{N}$ agents labeled using an index set $\mathcal{I}=\{1,\dots,\mfs{N}\}$, whose goal is to solve a consensus optimization program of the form:
\begin{equation}
\label{eq:optimal2}
\bm{\theta}^*=\arg\!\min_{\bm{\theta}\in\mathbb{R}^m} \sum_{i=1}^\mfs{N}f_i(\bm{\theta})
\end{equation}
where each $f_i:\mathbb{R}^m\to\mathbb{R}$ is a function which can only be accessed at agent $i\in\mathcal{I}$.

\begin{assumption}
    \label{as:regular:f}
    For all $i\in\mathcal{I}$, $f_i:\mathbb{R}^m\to\mathbb{R}$ is strongly convex with constant $\mu>0$ (See Appendix \ref{ap:convex}), and has Lipschitz continuous gradient with constant $L>0$.
\end{assumption}
Note that Assumption \ref{as:regular:f} ensures that the program in \eqref{eq:optimal2} has a unique optimal solution $\bm{\theta}^*\in\mathbb{R}^m$. Moreover, to obtain such optimal solution, agents cooperate by exchanging information with their neighbors. The communication network is modeled by a fixed undirected connected graph $\mathcal{G}=(\mathcal{I},\mathcal{E})$. Here, $\mathcal{I}$ is the vertex set, where each $i\in\mathcal{E}$ represents a concrete agent, and $\mathcal{E}\subseteq \mathcal{I}\times\mathcal{I}$ is the edge set where $(i,j)\in\mathcal{E}$ when agents $i,j\in\mathcal{I}$ can communicate. Moreover, we denote with $\mathcal{N}_i\subseteq\mathcal{I}$ the neighbor set for agent $i\in\mathcal{I}$.

Often, an algorithm to solve the program in \eqref{eq:optimal2} in a distributed fashion often involves a set of design parameters. The algorithm is parameter-free if the constraints for the parameters do not require prior information of the network or the cost function constants $(\mu,L)$ to ensure convergence towards the optimal solution. This work proposes a parameter-free solution to solve the program \eqref{eq:optimal2} based on the PHS structure. 

\subsection{Port-Hamiltonian systems}
\label{sec:phs}

A PHS is a particular class of dynamical system that can be described through the interaction of ports with power-conserving features. In particular, it has the form:
\begin{equation}
\label{eq:phs}
\begin{aligned}
    \dot{\mf{x}} &= \mf{F}(\mf{x})\nabla H(\mf{x}) + \bm{B}_I(\mf{x})\mf{u}_I + \bm{B}_C(\mf{x})\mf{u}_C
\end{aligned}
\end{equation}
where $\mf{x}(t)\in\mathbb{R}^n$ is the state variable. The system in \eqref{eq:phs} has four classes of ports. First, a storage port $(-\dot{\mf{x}},\nabla H(\mf{x}))$ with scalar storage function $H:\mathbb{R}^n\to\mathbb{R}$ also called the Hamiltonian energy. Second, a resistive structure associated with the symmetric part of $\mf{F}$, which satisfies $\mf{F}(\mf{x})+\mf{F}(\mf{x})^\top\preceq\mf{0}$. Third, an interaction port with $(\mf{u}_I, \mf{y}_I)$. This port usually represents the interaction of \eqref{eq:phs} with other systems. Fourth, a control port $(\mf{u}_C,\mf{y}_C)$. This port is used for feedback control.

It can be verified that \eqref{eq:phs} is passive in the sense that:
$$
\dot{H}(\mf{x})\leq \mf{y}_I^\top\mf{u}_I+\mf{y}_C^\top\mf{u}_C
$$
with outputs 
$$
\begin{aligned}
\mf{y}_I&=\mf{B}_I(\mf{x})^\top\nabla H(\mf{x}) \\ \mf{y}_C&=\mf{B}_C(\mf{x})^\top\nabla H(\mf{x}).
\end{aligned}
$$ 
As will become evident in subsequent sections, the passivity property helps design control feedback laws and obtain Lyapunov functions based on $H$. We refer the reader to \cite{Schaft2017} for more details on PHS.

\section{PHS-based optimization in continuous time}
\label{sec:phsopt}
\subsection{System proposal}
To go towards a parameter-free discrete time algorithm to solve \eqref{eq:optimal2}, we start by considering a continuous-time formulation based on a PHS. Assume that each agent $i\in\mathcal{I}$ stores a state variable $\mf{x}_i(t)\in\mathbb{R}^n$. The proposed dynamics for $\mf{x}_i(t)$ are given as:
\begin{equation}
\label{eq:general:MAS}
\begin{aligned}
    \dot{\mf{x}}_i &= \mf{F}_i\nabla H_i(\mf{x}_i) +\bm{\phi}_i(\nabla H_i(\mf{x}_i)) +\sum_{j\in\mathcal{N}_i}\mf{M}\left(\nabla H_i(\mf{x}_i)-\nabla H_j(\mf{x}_j)\right)
    \end{aligned}
\end{equation}
Where $\mf{F}_i,\mf{M}$ are constant matrices satisfiying $\mf{F}_i+\mf{F}_i^\top\preceq \mf{0}$ and  $\mf{M}+\mf{M}^\top\preceq \mf{0}$. Moreover, $H_i:\mathbb{R}^n\to\mathbb{R}$ is the storage function at agent $i$ and $\bm{\phi}_i:\mathbb{R}^n\to\mathbb{R}^n$ is a nonlinear vector field. Time dependence for the state was omitted in \eqref{eq:general:MAS} for readability. For the same reason, we will also omit state dependence on $H_i$ when there is no ambiguity. 

Note that the system in \eqref{eq:general:MAS} is a PHS since the symmetric part of $\mf{F}_i+|\mathcal{N}_i|\mf{M}$ is negative semi-definite. Additionally, it includes $|\mathcal{N}_i|$ interaction ports, each connected to one of the neighbors of the agent with input $\nabla H_j(\mf{x}_j)$ and input matrix $-\mf{M}$. Moreover, it features a local feedback control port with input $\bm{\phi}_i(\nabla H_i(\mf{x}_i))$ and input matrix $\mf{I}_n$. Note that the terms of the form $\nabla H_i-\nabla H_j$ in \eqref{eq:general:MAS} define an agreement error between neighboring agents and promote consensus.  

The remainder of this section aims to show that the PHS structure in \eqref{eq:general:MAS} naturally ensures convergence towards the equilibrium. Hence, given that stability is provided, the remaining degrees of freedom, namely $\mf{F}_i,\mf{M},\bm{\phi}_i, H_i$ can be used to design equilibrium to coincide with the solution of an optimization program of interest. To formalize these ideas, we make the following assumption.

\begin{assumption}[Well posedness]\label{as:well:posedness}
Let $H_i, \mf{F}_i, \mf{M}, {\bm{\phi}}_i$ be such that:
\begin{enumerate}
    \item  The PHS in \eqref{eq:general:MAS} has a unique solution for all $t\geq 0, \mf{x}_i(0)\in\mathbb{R}^n, i\in\mathcal{I}$.
    \item There exist unique constant vectors $\{\mf{x}_i^*\}_{i\in\mathcal{I}}$ such that
\begin{equation}
\label{eq:general:equilibrium}
\begin{aligned}
    \mf{0} &= \mf{F}_i\nabla H_i(\mf{x}_i^*) +\bm{\phi}(\nabla H_i(\mf{x}_i^*)) +\sum_{j\in\mathcal{N}_i}\mf{M}\left(\nabla H_i(\mf{x}_i^*)-\nabla H_j(\mf{x}_j^*)\right)
    \end{aligned}
\end{equation}
\end{enumerate}
\end{assumption}

To show convergence towards the unique equilibrium point satisfying \eqref{eq:general:equilibrium}, we write the joint dynamics of all agents incorporating \eqref{eq:general:MAS} for all $i\in\mathcal{I}$. Let the joint storage function 
$$
H(\mf{x}) = \sum_{i\in\mathcal{I}}H_i(\mf{x}_i),$$ the joint output $\mf{y} = [\nabla H_i(\mf{x}_i)]_{i\in\mathcal{I}}$ and the joint input $\bm{\phi}(\mf{y}) = [\bm{\phi}_i(\nabla H_i(\mf{x}_i))]_{i\in\mathcal{I}}$ with joint state $\mf{x}=[\mf{x}_i]_{i\in\mathcal{I}}$. Moreover, let  $\mf{F}=\text{blockdiag}(\mf{F}_1,\dots,\mf{F}_\mfs{N})$. Hence, \eqref{eq:general:MAS} can be written in compact form as follows. 
\begin{equation}
\label{eq:mas:total}
\dot{\mf{x}} = (\mf{F}+\mf{L}\otimes \mf{M})\nabla H(\mf{x}) + \bm{\phi}(\nabla H(\mf{x}))
\end{equation}
where $\mf{L} = \text{\normalfont diag}(\mf{A}\mathds{1})-\mf{A}$ is the standard Laplacian matrix of the graph $\mathcal{G}$ associated with the $0-1$ adjacency matrix $\mf{A}$. 

One of the most important features of \eqref{eq:mas:total} inherited from the PHS structure is its passivity property concerning the input-output pair $(\bm{\phi}(\nabla H(\mf{x})),\mf{y})$. 
\begin{proposition}
\label{prop:passive}
The trajectories of \eqref{eq:mas:total} satisfy
$$
\dot{H}(\mf{x})\leq \mf{y}^\top\bm{\phi}(\nabla H(\mf{x}))
$$
\end{proposition}
\begin{proof}
    Take arbitrary $\mf{v}\in\mathbb{R}^{\mfs{N}m}$ with $\mf{v}\neq \mf{0}$. Recall $\mf{M}+\mf{M}^\top\preceq \mf{0}$ and $\mf{F}+\mf{F}^\top\preceq \mf{0}$ by assumption. Define $\mf{w}=(\mf{E}^\top\otimes \mf{I}_m)\mf{v}$ where $\mf{E}$ is an incidence matrix of $\mathcal{G}$ satisfying $\mf{L}=\mf{E}\mf{E}^\top$. Hence,
    
    $$
    \begin{aligned}
    \mf{v}^\top(\mf{F}+\mf{L}\otimes\mf{M})\mf{v}&= \mf{v}^\top\mf{F}\mf{v}+ \mf{v}^\top(\mf{E}\mf{E}^\top\otimes \mf{M})\mf{v}\\
    &\leq \mf{v}^\top(\mf{E}\otimes \mf{I}_m)(\mf{I}\otimes \mf{M})(\mf{E}^\top\otimes \mf{I}_m)\mf{v}\\
    &\leq \mf{w}^\top(\mf{I}\otimes\mf{M})\mf{w}\leq 0
    \end{aligned}
    $$
    Therefore, $\mf{F}+\mf{L}\otimes\mf{M}$ has symmetric part negative semi-definite. Hence, the proposition follows by taking $\dot{H}=\nabla H^\top \dot{\mf{x}}$ along \eqref{eq:mas:total}.
\end{proof}
As discussed in Section \ref{sec:phs}, passivity can be used to construct a Lyapunov function, ensuring convergence to the equilibrium. This idea is formalized in the following.
\begin{corollary}
\label{cor:convergence}
Let Assumption \ref{as:well:posedness} and consider \eqref{eq:mas:total}. Moreover, assume that $H(\mf{x})$ is strongly convex. Finally, assume that $\bm{\phi}_i$ satisfy the incremental passivity property:
\begin{equation}
\label{eq:increasing:passivity}
(\mf{u}-\mf{v})^\top(\bm{\phi}_i(\mf{u})-\bm{\phi}_i(\mf{v}))\leq 0
\end{equation}
for arbitrary $\mf{u}, \mf{v}\in\mathbb{R}^n$. with equality only when $\mf{u}=\mf{v}$. Then, the unique equilibrium of \eqref{eq:mas:total} is globally asymptotically stable.
\end{corollary}
\begin{proof}
Let $\mf{x}^*$ denote the equilibrium of \eqref{eq:mas:total}. To simplify the notation, let $\nabla H^*=\nabla H(\mf{x}^*)$.
Set 
\begin{equation}
\label{eq:lyapunov}
    V_{\mf{x}^*}(\mf{x}) = H(\mf{x}) - H(\mf{x}^*) - \nabla H(\mf{x}^*)^\top(\mf{x}-\mf{x}^*)
\end{equation}
as a Lyapunov function candidate. Note that $V_{\mf{x}^*}(\mf{x})$ satisfy the required conditions for this purpose. First, $V_{\mf{x}^*}(\mf{x}^*) = 0$. Second, since $H$ is strongly convex, then Lemma \ref{le:lyapunov} in Appendix \ref{ap:convex} is used to conclude that $V_{\mf{x}^*}(\mf{x})$ is always positive and radially unbounded. Denote with $\mf{S}=(\mf{F}+\mf{L}\otimes \mf{M})$. Then,
    $$
    \begin{aligned}
&\dot{V}_{\mf{x}^*} 
=\left(\nabla H-\nabla H^*\right)^\top\dot{\mf{x}} \\
& = \left(\nabla H-\nabla H^*\right)^\top(\dot{\mf{x}} - \mf{S}\nabla H^* - \bm{\phi}(\nabla H^*)) \\
&=\left(\nabla H-\nabla H^*\right)^\top(\mf{S}\nabla H + \bm{\phi}(\nabla H ) - \mf{S}\nabla H^* - \bm{\phi}(\nabla H^*))\\
&=  \left(\nabla H -\nabla H^*\right)^\top \mf{S} \left(\nabla H -\nabla H^*\right)\\
&+  \left(\nabla H -\nabla H^*\right)^\top(\bm{\phi}(\nabla H )-\bm{\phi}(\nabla H^*)) \leq 0
    \end{aligned}
    $$
    where we used the fact that $\mf{S}+\mf{S}^\top\preceq \mf{0}$ from the proof of Proposition \ref{prop:passive}. Moreover, we used the incremental passivity property from \eqref{eq:increasing:passivity} in the last step. Furtheremore, note that $\dot{V}_{\mf{x}^*}=0$ only when $\nabla H(\mf{x}) = \nabla H(\mf{x}^*)$, which occurs only at the equilibrium point. Henceforth, $\dot{V}_{\mf{x}^*}<0$ for all $\mf{x}\neq \mf{x}^*$. Therefore, $\mf{x}^*$ is a globally asymptotically stable equilibrium.
\end{proof}
\subsection{Matching the equilibrium with the solution of an optimization program}
Given that convergence of the PHS system in \eqref{eq:general:MAS} is ensured by Proposition \ref{cor:convergence}, it only remains to match the equilibrium equations \eqref{eq:general:equilibrium} with the solution of the consensus optimization program in \eqref{eq:optimal2}. This can be performed by an appropriate design of the degrees of freedom we have in $\mf{F}_i, \mf{M}, \bm{\phi}_i, H_i$. To provide intuition, we start with the centralized setting.
\begin{example}[Centralized optimization]\label{ex:single}
{\normalfont
Consider \eqref{eq:optimal2} with $\mfs{N}=1$ and $f:=f_1$, $\mf{x}:=\mf{x}_1$. For an equilibrium $\mf{x}^*$ of \eqref{eq:mas:total} to coincide with $\bm{\theta^*}$, the first order optimality constraints must be met, namely,
$
\mf{0} = \nabla f(\mf{x}^*) 
$.
Comparing with the equilibrium conditions \eqref{eq:general:equilibrium}, this condition can be achieved with either one of the following choices:\\[-0.8em]
\begin{enumerate}
    \item[a)] $H_1(\mf{x}) = \mf{x}^\top\mf{x}/2$, $\mf{F}_1=\mf{M}=\mf{0}$, $\mf{\phi}_1(\mf{x})=-\nabla f(\mf{x})$.\\[-0.8em]
    \item[b)] $H_1(\mf{x}) = f(\mf{x})$, $\mf{F}_1=-\mf{I}$, $\mf{M}=\mf{0}$, $\mf{\phi}_1(\mf{x})=\mf{0}$.\\[-0.8em]
\end{enumerate}
In both cases, the same system results from \eqref{eq:general:MAS}, which coincides with the standard gradient flow,
\begin{equation}
\label{eq:grad:flow}
\dot{\mf{x}} = -\nabla f(\mf{x}).
\end{equation}
Note that strong convexity of $f$ imply that the incremental passivity property \eqref{eq:increasing:passivity} is satisfied for $\mf{\phi}_1(\mf{x})=-\nabla f(\mf{x})$ due to Proposition \ref{prop:boyd} in Appendix \ref{ap:convex}. Thus, convergence of trajectories of the standard gradient flow \eqref{eq:grad:flow} towards $\mf{x}^*$ follows from Corollary \ref{cor:convergence}.
}
\end{example}
Now, we provide an example aiming to solve \eqref{eq:optimal2}.

\begin{example}[Unconstrained distributed optimization]\label{ex:distro}
{\normalfont 
Consider \eqref{eq:optimal2} with arbitrary $\mfs{N}\in\mathbb{N}$ and connected $\mathcal{G}$. The first order optimimality conditions for the global optimum of the program \eqref{eq:optimal2} are $\mf{0}=\sum_{i\in\mathcal{I}}\nabla f_i(\bm{\theta}^*)$. However, this condition can be expressed in a more appropriate distributed fashion as:
\begin{equation}
\label{eq:conditions}
\mf{0}=\sum_{i\in\mathcal{I}}\nabla f_i(\mf{q}_i^*), \ \ \mf{q}_1^*=\dots=\mf{q}_\mfs{N}^*
\end{equation}
where we introduce auxiliary variables $\mf{q}_i^*=\bm{\theta}^*\in\mathbb{R}^{m}$. These include a condition for the gradients and a condition for consensus. Setting $\mf{q}^*=[\mf{q}_i^*]_{i\in\mathcal{I}}$, the conditions in \eqref{eq:conditions} at the equilibrium can be written as:
\begin{equation}
\label{eq:conditions2}
\mf{0}=(\mathds{1}^\top_\mfs{N}\otimes \mf{I}_m)\nabla f(\mf{q}^*),\ \  \mf{0}=(\mf{L}\otimes \mf{I}_m)\mf{q}^*
\end{equation}
which implies $\mf{q}^*=(\mathds{1}_\mfs{N}\otimes\bm{\theta}^*)$ and with $f(\mf{q}^*)=\sum_{i\in\mathcal{I}} f_i(\mf{q}_i^*)$. Having two sets of constraints motivates the addition of a second set of auxiliary variables $\mf{p}_i\in\mathbb{R}^{m}$ and partition $\mf{x}_i=[\mf{q}_i^\top,\mf{p}_i^\top]^\top\in\mathbb{R}^n$ with $n=2m$. The PHS equilibrium equations in \eqref{eq:general:equilibrium} for each partition will be used to match either one of the two conditions \eqref{eq:conditions2}. This can be performed by setting 
\begin{equation}
\label{eq:M}
\begin{aligned}
\mf{F}_i=\mf{0}_{n\times n},& \quad \mf{H}_i(\mf{x}_i)=\frac{1}{2}\mf{x}_i^\top\mf{x}_i \\
\bm{\phi}_i(\mf{x}_i)=\begin{bmatrix}-\nabla f_i(\mf{q}_i)\\ \mf{0}_m\end{bmatrix},& \quad \mf{M} = \begin{bmatrix}
-1 & -1 \\
 1 & 0
\end{bmatrix}\otimes \mf{I}_m
\end{aligned}
\end{equation}
Henceforth, from \eqref{eq:general:equilibrium}, the equilibrium for the partitioned states can be written as
\begin{equation}
\label{eq:equalibrium:example}
\begin{aligned}
\mf{0} &= -(\mf{L}\otimes \mf{I}_m)\mf{q}^* -(\mf{L}\otimes \mf{I}_m)\mf{p}^* -\nabla {f}(\mf{q}^*) \\
\mf{0} &= (\mf{L}\otimes \mf{I}_m)\mf{q}^*
\end{aligned}
\end{equation}
from which the optimally conditions in \eqref{eq:conditions2} follow by multiplying \eqref{eq:equalibrium:example} with $(\mathds{1}^\top_\mfs{N}\otimes \mf{I}_m)$ from the left and using $\mathds{1}^\top_\mfs{N}\mf{L}=\mf{0}$. The resulting system at each agent from \eqref{eq:general:MAS} is
\begin{equation}
    \label{eq:system:unconstrained}
\begin{aligned}
\dot{\mf{q}}_i &= -\sum_{i\in\mathcal{I}}(\mf{q}_i-\mf{q}_i)-\sum_{i\in\mathcal{I}}(\mf{p}_i-\mf{p}_i)-\nabla f_i(\mf{q}_i)\\
\dot{\mf{p}}_i &= \sum_{i\in\mathcal{I}}(\mf{q}_i-\mf{q}_i)
\end{aligned}
\end{equation}
}
\end{example}

\begin{remark}
Note that in all previous examples, global convergence is ensured by Corollary \ref{cor:convergence}. The only condition to be verified is existence and uniqueness of solutions required in Assumption \ref{as:well:posedness}. Fortunately, this is easily verified, since Assumption \ref{as:regular:f} for the functions $f_i$ ensure the Lipschitz property of the right-hand sides of \eqref{eq:grad:flow} and \eqref{eq:system:unconstrained}. Henceforth, existence and uniqueness of solutions follow in each case.
\end{remark}

\section{PHS-based optimization in discrete time}
\label{sec:discrete:time}
\subsection{Discrete gradient-based PHS discretization}
The continuous time system in \eqref{eq:general:MAS} can be implemented in discrete time by taking advantage of the PHS structure one more time. For instance, we make the following changes:
\begin{equation}
\label{eq:changes}
\begin{aligned}
i)&\quad \dot{\mf{x}}_i\to \frac{\mf{x}^+_i-\mf{x}_i}{\tau}\\
ii)&\quad \nabla H_i(\mf{x}_i)\to\overline{\nabla} H_i(\mf{x}_i,\mf{x}_i^+)
\end{aligned}
\end{equation}
where we denote with $\mf{x}_i=\mf{x}_i[k]$ at an arbitrary time step $k\in\mathbb{N}$ and $\mf{x}_i^+=\mf{x}_i[k+1]$ the updated state at the next time step, and the object $\overline{\nabla} H_i$ is defined later. Note that the change \eqref{eq:changes}-\textit{i)} corresponds to the finite difference approximation of the derivative $\dot{\mf{x}}_i$ between two time instants $t=k\tau$ and $t=(k+1)\tau$, where $\tau>0$ represents a step size parameter. If this was the only change to be applied, the discretization corresponds to the standard forward Euler method. For instance, applying this change to \eqref{eq:grad:flow}, leads to the standard gradient descent method
$$
\mf{x}^+ = \mf{x} - \tau\nabla f(\mf{x})
$$
from which $\tau>0$ can be interpreted as a learning rate parameter. In general, upper bound constraints need to be applied to $\tau$ to ensure sufficiently a small $\tau$ which leads to convergence.

In order to circumvent any constraint to $\tau$, we introduce \eqref{eq:changes}-\textit{ii}), such that our discretization no longer corresponds to a forward Euler step, where the symbol $\overline{\nabla}$ denotes the discrete gradient defined as follows:

\begin{definition}\cite{discretegrad}
Let $g:\mathbb{R}^n\to\mathbb{R}$ be a scalar differentiable function. Hence, $\overline{\nabla}g: \mathbb{R}^n\times \mathbb{R}^n\to\mathbb{R}$ is defined by the properties:
\begin{equation}
\label{eq:prop:dd}
    \begin{aligned}
    &\overline{\nabla} g(\mf{u},\mf{v})^\top(\mf{v}-\mf{u}) = g(\mf{u})-g(\mf{v}) \\
    &\lim_{\mf{v}\to\mf{u}}\overline{\nabla} g(\mf{u},\mf{v}) = \nabla g(\mf{u})
    \end{aligned}
\end{equation}
\end{definition}
While there are many possible ways to construct a discrete gradient, in this work, we consider 
$$
\overline{\nabla} g(\mf{u},\mf{v})=\int_0^1\nabla g((1-s)\mf{u}+s\mf{v})\text{d}s
$$
which for the case of quadratic functions $H_i(\mf{x}_i) = \mf{x}_i^\top\mf{x}_i/2$, results in the average $$\overline{\nabla} H_i(\mf{x}_i,\mf{x}_i^+) =\frac{\mf{x}_i+\mf{x}_i^+}{2}$$

With these changes, the discrete time equivalent of the PHS \eqref{eq:general:MAS} corresponds to: 
\begin{equation}
\label{eq:discrete:MAS}
\begin{aligned}
     \frac{\mf{x}^+_i-\mf{x}_i}{\tau} &= \mf{F}_i\overline{\nabla} H_i(\mf{x}_i,\mf{x}_i^+) +\bm{\phi}_i(\overline{\nabla} H_i(\mf{x}_i,\mf{x}_i^+)) \\
    &+\sum_{j\in\mathcal{N}_i}\mf{M}\left(\overline{\nabla} H_i(\mf{x}_i,\mf{x}_i^+)-\overline{\nabla} H_j(\mf{x}_j,\mf{x}_j^+)\right)
    \end{aligned}
\end{equation}
Note that the discrete time system in \eqref{eq:discrete:MAS} imposes an implicit equation of $\mf{x}_i^+$ at each time step. Hence, we require an additional regularity condition, analogous to existence and uniqueness of solutions in the second part of Assumption \ref{as:well:posedness} which, similar to \cite{Macchelli2023,discretegrad}, can be fulfilled for $\bm{\phi}_i, \overline{\nabla}H_i$ Lipschitz:
\begin{assumption}
\label{as:implicit1}
Given $\{\mf{x}_i\}_{i\in\mathcal{I}}$, there exists unique $\{\mf{x}_i^+\}_{i\in\mathcal{I}}$ complying \eqref{eq:discrete:MAS}.
\end{assumption}

As discussed before, a parameter-free algorithm refers to the property of \eqref{eq:discrete:MAS} not to compromise its stability features regardless of the step size $\tau>0$. In this sense, the advantage of \eqref{eq:discrete:MAS} is twofold. First, note that since $$\overline{\nabla} H_i(\mf{x}_i^*,\mf{x}_i^+)=\nabla H_i(\mf{x}_i^*)$$ by construction, then the equilibrium of \eqref{eq:discrete:MAS} is the same as the one from \eqref{eq:general:MAS}, namely the unique point $\mf{x}_i^*$ complying \eqref{eq:general:equilibrium}. Second, the stability proof for \eqref{eq:discrete:MAS} is practically the same as \eqref{eq:general:MAS} regardless of $\tau>0$ as is formalized in the following:
\begin{theorem}
\label{prop:discrete:convergence}
Let Assumption \ref{as:well:posedness}-2), Assumption \ref{as:implicit1} and consider \eqref{eq:discrete:MAS}. Moreover, assume that $H_i, \bm{\phi}_i$ satisfies the conditions described in Corollary \ref{cor:convergence}. Then, the unique equilibrium of \eqref{eq:discrete:MAS} is globally asymptotically stable for any $\tau>0$.
\end{theorem}
\begin{proof}
First, write the discrete time system \eqref{eq:discrete:MAS} in compact form, similarly as done before with the continuous time PHS from \eqref{eq:mas:total}:
\begin{equation}
\label{eq:mas:total:discrete}
\frac{\mf{x}^+-\mf{x}}{\tau} = (\mf{F}+\mf{L}\otimes \mf{M})\overline{\nabla} H(\mf{x},\mf{x}^+) + \bm{\phi}(\overline{\nabla} H(\mf{x},\mf{x}^+))
\end{equation}
with $\mf{x}=[\mf{x}_i]_{i\in\mathcal{I}}, \mf{x}^+=[\mf{x}_i^+]_{i\in\mathcal{I}}$. Moreover, denote $\mf{S}=(\mf{F}+\mf{L}\otimes \mf{M})$.  Similarly, as in the proof Corollary \ref{cor:convergence}, set $V_{\mf{x}^*}(\mf{x})$ defined in \eqref{eq:lyapunov} as a Lyapunov function candidate. Then, due to the definition of the discrete gradient, it follows that
    $$
    \begin{aligned}
    &V(\mf{x}^+) - V(\mf{x}) = \overline{\nabla}V(\mf{x},\mf{x}^+)^\top( {\mf{x}}^+ - {\mf{x}} ) \\
    &=\tau\overline{\nabla}V(\mf{x},\mf{x}^+)^\top\left( \mf{S}\overline{\nabla}H( \mf{x},\mf{x}^+) + \bm{\phi}\left(\overline{\nabla}H(\mf{x},\mf{x}^+)\right) \right). 
    \end{aligned}
    $$
Then, use 
$$
\overline{\nabla}V(\mf{x},\mf{x}^+) = \overline{\nabla}H(\mf{x},\mf{x}^+) - \overline{\nabla}H(\mf{x}^*,\mf{x}^*)
$$
since $\overline{\nabla}H(\mf{x}^*,\mf{x}^*)=\nabla H(\mf{x}^*)$. From this point, the proof follows the same steps as those in the proof of Corollary \ref{cor:convergence} with the discrete gradient of $H$ in place of the gradient to show that $ V(\mf{x}^+) - V(\mf{x})\leq 0$, with equality only when $\mf{x}=\mf{x}^*$, concluding the proof.
\end{proof}
The formulation presented so far from \eqref{eq:discrete:MAS} is already explicitly useful for parameter-free optimization in a non-distributed setting, as evidenced in the following example.
\begin{example}
{\normalfont
Consider Example \ref{ex:single} and apply the two design choices of $\mf{F}_1, \mf{M}, H_1, \bm{\phi}_1$ on \eqref{eq:discrete:MAS}. Hence, two possible discrete time systems for $\mf{x}:=\mf{x}_1$ are obtained:

\begin{align}
\frac{\mf{x}^+-\mf{x}}{\tau} &= -\nabla f\left(\frac{\mf{x}^++\mf{x}}{2}\right)\label{eq:flow1} \\
\frac{\mf{x}^+-\mf{x}}{\tau} &= -\overline{\nabla} f\left({\mf{x},\mf{x}^+}\right) \label{eq:flow2}
\end{align}
Both systems are ensured to converge to the correct equilibrium regardless of $\tau>0$. While \eqref{eq:flow2} was discussed in \cite{discretegrad}, requiring computing the discrete gradient of the cost function, which can be expensive, \eqref{eq:flow1} uses the usual gradient, which is more standard and readily available.
}    
\end{example}

To implement the discrete time system in \eqref{eq:discrete:MAS}, each agent $i\in\mathcal{I}$ need to solve its implicit equations for $\mf{x}_i^+$. This can be challenging since the implicit equations also depend on the neighbor values $\mf{x}_j^+$. While a central station may solve for $\{\mf{x}_i^+\}_{i\in\mathcal{I}}$ at each time step,  distributed equation solvers can be used at each time step. However, such an algorithm requires additional communication overhead. Henceforth, we propose an additional novel modification to the discretization methodology to remove the dependence on the neighbors $\{\mf{x}_i^+\}_{i\in\mathcal{I}}$ in the implicit equations at agent $i\in\mathcal{I}$, in a consensus optimization problem.

\subsection{The MID method: fully-distributed PHS in discrete time}

In the following, we show a simple modification of \eqref{eq:discrete:MAS}, which results in a fully distributed algorithm that does not require additional distributed equation solvers. 
Aiming to obtain a solution of the program in \eqref{eq:optimal2}, we set \eqref{eq:discrete:MAS} with the choice $H_i(\mf{x}_i) = \mf{x}_i^\top\mf{x}_i/2, \mf{F}_i=\mf{0}$ which coincide with the design in \eqref{eq:system:unconstrained}. These choices result in:
\begin{equation}
\label{eq:discrete:MAS:special}
\begin{aligned}
     \frac{\mf{x}^+_i-\mf{x}_i}{\tau} & =\sum_{j\in\mathcal{N}_i}\mf{M}\left(\frac{\mf{x}_i+\mf{x}_i^+}{2} -\frac{\mf{x}_j+\mf{x}_j^+}{2} \right)+\bm{\phi}_i\left(\frac{\mf{x}_i+\mf{x}_i^+}{2}\right) 
    \end{aligned}
\end{equation}
We remove the occurrence of $\mf{x}_j^+$ in the previous equation by making the following modification, which we refer to as the Mixed Implicit Discretization (MID):
\begin{equation}
\label{eq:discrete:MAS:MID}
\begin{aligned}
     \frac{\mf{x}^+_i-\mf{x}_i}{\tau} &=\sum_{j\in\mathcal{N}_i}\mf{M}\left(\mf{x}_i^+ -\mf{x}_j\right)+\bm{\phi}_i\left(\frac{\mf{x}_i+\mf{x}_i^+}{2}\right),
    \end{aligned}
\end{equation}
where the disagreement difference is computed with the future state for agent $i$ and the present state for agent $j$. This modification allows each agent to locally solve \eqref{eq:discrete:MAS:MID} without requiring further communication between its neighbors besides sharing $\mf{x}_i$ once at each time step. It is important to note that this modification has the same equilibrium point as \eqref{eq:discrete:MAS:special}. Henceforth, MID applied to the continuous time PHS in \eqref{eq:system:unconstrained} results in:
    \begin{equation}
    \label{eq:main:algorithm}
    \begin{aligned}
    \frac{\mf{q}_i^+-\mf{q}_i}{\tau} &= -\sum_{j\in\mathcal{N}_i}(\mf{q}_i^+-\mf{q}_j+\mf{p}_i^+-\mf{p}_j) - \nabla f_i\left(\frac{\mf{q}_i^++\mf{q}_i}{2}\right)\\
    \frac{\mf{p}_i^+-\mf{p}_i}{\tau} &= \sum_{j\in\mathcal{N}_i}(\mf{q}_i^+-\mf{q}_j)\\
    \end{aligned}
    \end{equation}

Some important features of \eqref{eq:main:algorithm} are described in the following result.

\begin{theorem}
\label{th:main}
Let Assumptions \ref{as:regular:f} and \ref{as:well:posedness}-2) hold and consider \eqref{eq:main:algorithm}. Let the matrix definitions for $\bm{\Lambda}(\tau), \mf{S}_r(\tau)$ in \eqref{eq:matrices} and \eqref{eq:S:gamma} placed in Appendix \ref{ap:proof} for readability. Then, given $\tau>0$, the unique equilibrium of \eqref{eq:main:algorithm} is globally asymptotically stable if the following Linear Matrix Inequality (LMI) is satisfied:
\begin{subequations}
\label{eq:lmi}
\begin{align}
    &\mf{P}(\tau) := \begin{bmatrix}
    \bm{\Lambda}(\tau) & \mf{P}_{12}\\
    \mf{P}_{12}^\top & \mf{P}_{22}
    \end{bmatrix}\succ \mf{0}\label{eq:lmi:P}\\
    &\begin{bmatrix}
    \mf{U} & \mf{P}_{12} \\
    \mf{P}_{12}^\top & \mf{I}
    \end{bmatrix}\succeq \mf{0}, \quad u>0\label{eq:lmi:U}\\
    &\mf{P}(\tau)\mf{S}_r(\tau)+\mf{S}_r(\tau)^\top\mf{P}(\tau)+\mf{S}_\gamma(\tau, \varepsilon)\preceq -u\begin{bmatrix}
        \mf{I}, \mf{0}\\
        \mf{0}, \mf{0}
    \end{bmatrix}\label{eq:lmi:S}
    \end{align}
\end{subequations}
for some $\varepsilon\geq 0$ and decision variables $\mf{U}, \mf{P}_{12}, \mf{P}_{22}\in\mathbb{R}^{\mfs{N}m\times \mfs{N}m}, u\in\mathbb{R}$.
\end{theorem}
The proof for Theorem \ref{th:main} is given in Appendix \ref{ap:proof}.
\begin{remark}
Note that Theorem \ref{th:main} does not rely on an assumption similar to Assumption \ref{as:implicit1}. Instead, we show the existence of a unique solution to the implicit equation in \eqref{eq:main:algorithm} directly in the proofs.
\end{remark}
The following are some consequences of Theorem \ref{th:main}. While LMI formulations are usually conservative, we provide a convergence result that holds regardless of the value of $\tau>0$.
\begin{corollary}
\label{cor:special}
    Let the conditions of Theorem \ref{th:main} hold. Denote with $\mf{x}^*=[\mf{x}_i]_{i\in\mathcal{I}}$ the the unique equilibrium of \eqref{eq:main:algorithm}. Let $\mf{A}$ be the $0-1$ adjacency matrix of $\mathcal{G}$ and $\mf{D}=\text{diag}(\mf{A}\mathds{1})$ denote the degree matrix. Then, if
    \begin{equation}
    \label{eq:special}
        \mf{D}^2-\mf{A}^2\succeq \mf{0}
    \end{equation}
    it follows that $\mf{x}^*$ is globally asymptotically stable for all $\tau>0$. In particular, \eqref{eq:special} holds if $\mathcal{G}$ is a cycle or a fully connected graph of arbitrary size. 
\end{corollary} 
From this point, the proofs for all Corollaries are given in Appendix \ref{ap:cor}.

In the following, we show that even if the graph does not comply with \eqref{eq:special}, one can find a guaranteed bound for the time step $\tau>0$, which holds for any connected undirected graphs.
\begin{corollary}
\label{cor:bound}
    Let the conditions of Theorem \ref{th:main} hold. Denote with $\mf{x}^*=[\mf{x}_i]_{i\in\mathcal{I}}$ the unique equilibrium of \eqref{eq:main:algorithm}.  If $$
    \tau<\frac{\mu}{\|\mf{D}^2-\mf{A}^2\|}
    $$
    then $\mf{x}^*$ is globally asymptotically stable for all connected undirected graphs $\mathcal{G}$. 
\end{corollary}

\section{Numerical examples}
\label{sec:simulations}
In this section, we test the MID-based algorithm in \eqref{eq:main:algorithm} in several scenarios and compare it with other works in the literature. For the implementation, we used the function \texttt{fsolve} in MATLAB so that each agent can solve its local implicit equation at each time step. 

We start by testing the proposal's performance in a practical data analytics problem, similar to that in \cite[Section 6]{guido2023}. Each agent in $i\in\mathcal{I}$ is equipped with $d_i\in\mathbb{N}$ points $\mf{m}_{i,1},\dots,\mf{m}_{i,d_i}\in\mathbb{R}^{m-1}$ with binary labels $l_{i,1},\dots,l_{i,d_i}\in\{-1,1\}$. The cost function is then of the form \eqref{eq:optimal2} with
$$
f_i(\bm{\theta})=\sum_{\mu=1}^{d_i}\log(1+\exp(-l_{i,\mu}(\tilde{\bm{\theta}}^\top\mf{m}_{i,\mu}+\theta)))+\frac{C\|\bm{\theta}\|^2}{2\mfs{N}}
$$
where $\theta\in\mathbb{R}$ and $\bm{\theta}=[\tilde{\bm{\theta}}^\top,\theta]^\top$ is the optimization variable, and $C>0$ is a regularization parameter, making $f_i(\bm{\theta})$ strongly convex. For this example, we picked $\mfs{N}=10$ agents, connected in an Erd\H{o}s-Réyni graph $\mathcal{G}$ with parameter 0.4 \cite{guido2023}. Moreover, we set $m=3, d_i=10$ for all agents and $C=0.1$, with randomly chosen $\{\mf{m}_{i,\mu}\}_{\mu=1}^{d_i}$. In addition, initial conditions for the state at each agent were chosen randomly.

We start by comparing the convergence speed between MID and explicit forward Euler discretization by performing several simulations with $200$ values of $\tau$ uniformly distributed in the interval $[0,10]$. The figure of merit for the convergence speed is computed as the number of iterations $K_B\in\mathbb{N}$ required such that the condition $\|\mf{q}[k]-(\mathds{1}\otimes\mf{I}_m)\bm{\theta}^*\|\leq B$ is complied from that point forward. Here, we choose $B=10^{-6}$ as an accuracy bound representing when the error between the system output and true optimal are sufficiently close for illustrative purposes. The results are shown in Figure \ref{fig:speed}, where it is observed that both methods achieve similar convergence speed for small $\tau$. However, near $\tau\approx 0.77$, Euler discretization becomes unstable, while the sampling step can be increased further for MID. This allows us to decrease $K_B$ for our method, reaching a minimum near $\tau\approx 4$. From this point, $K_B$ increases with $\tau$. This shows that allowing bigger values of $\tau$ can increase convergence speed compared to classical discretization methods, for which increasing $\tau$ leads to divergence.

Moreover, we compare with other methods in the literature for consensus distributed optimization. In particular, for discrete time, we compare with the Discrete Time Gradient Tracking (DTGT) \cite{Notarstefano2019} and the DIGing method \cite{nedic2017achieving}. In both cases, a step size parameter $\tau$ is used, which needs to be sufficiently small. Moreover, we compare with Euler discretization of continuous time approaches such as the Continuous Time  Gradient Tracking (CTGT)\cite{guido2023}, and the coordination algorithm (COOR) from \cite{kia2015distributed}. In both cases, we assume communication at each time step for fairness in all methods. Additionally, the time step for these methods is assigned to $\tau$. Finally, we also compare with Euler discretization of the PHS system in \eqref{eq:general:MAS}.

For comparison, we used quadratic cost functions, which allows us to verify a less conservative LMI for any graph, as stated in the following result:

\begin{corollary}
\label{cor:linear}
    Let the conditions of Theorem \ref{th:main} hold. Denote with $\mf{x}^*=[\mf{x}_i]_{i\in\mathcal{I}}$ the unique equilibrium of \eqref{eq:main:algorithm}.  Given $\tau>0, \mathcal{G}$, as well as $f_i(\mf{x}_i)=\mf{x}_i^\top\mf{H}_i\mf{x}_i/2+\mf{b}^\top_i\mf{x}_i$, $\mf{x}^*$ is globally asymptotically stable if the LMI comprised by \eqref{eq:lmi:P}, $u>0$ and
    \begin{equation}
    \label{eq:lmi:linear}
    \mf{P}(\tau)\mf{S}_r(\tau)+\mf{S}_r(\tau)^\top\mf{P}(\tau)+\mf{S}_h(\tau)
    \preceq -u\begin{bmatrix}
        \mf{I}, \mf{0}\\
        \mf{0}, \mf{0}
    \end{bmatrix}
    \end{equation}
    is satisfied, with $\mf{S}_h(\tau)$ defined in \eqref{eq:Sh}.   
\end{corollary}

Moreover, for this comparison, we looked for the best value of $\tau>0$, leading to the fastest behavior in each method tested. These values of $\tau$ were found by trial and error by increasing its value until divergence and keeping the best one in terms of speed. For MID, Euler discretization, DTGT, CTGT, DIGing, and COOR, the values of $\tau$ resulted in approximately $3.78, 0.71, 0.58, 0.65, 0.05, 0.078$ respectively. Trajectories for each case are shown in Figure \ref{fig:comparison:smalltau}, where all methods converge to the global optimum, with MID having a clear advantage in speed. 

In contrast, Figure \ref{fig:comparison:bigtau} shows the same scenario but now with $\tau=10$, showing that all other methods diverge while our proposal manages to converge. For both $\tau=3.78, 10$, we verified the feasibility of the LMI in \eqref{eq:lmi:linear}, which guarantees convergence.

To further stress our proposal, we performed simulations with $\tau=1,10,100,1000$. In all cases, the LMI in \eqref{eq:lmi:linear} was feasible. The convergence results are shown in Figure \ref{fig:multipletau}, where it is observed that even in the extreme case of $\tau=1000$, our proposal converges despite a slower convergence rate.

\begin{figure}
\centering
\includegraphics[width=0.5\textwidth]{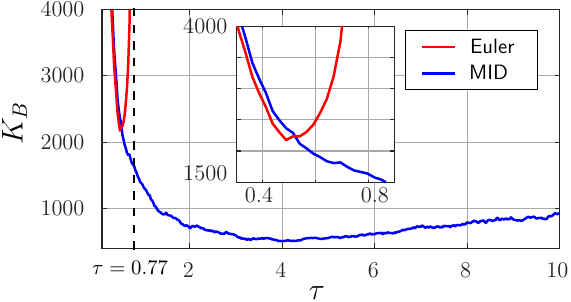}
\caption{Comparison of the number of iterations $K_B$ required to reach $\|\mf{q}[k]-(\mathds{1}\otimes\mf{I}_m)\bm{\theta}^*\|\leq B$ with $B=10^{-6}$.  The proposed MID allows the reduction of the number of iterations by increasing $\tau$, with an optimal value around $\tau=4$. In contrast, increasing $\tau$ for Euler discretization does not feature such great reduction and leads to divergence at around $\tau=0.77$.}
\label{fig:speed}
\end{figure}

\begin{figure}
\centering
\includegraphics[width=0.5\textwidth]{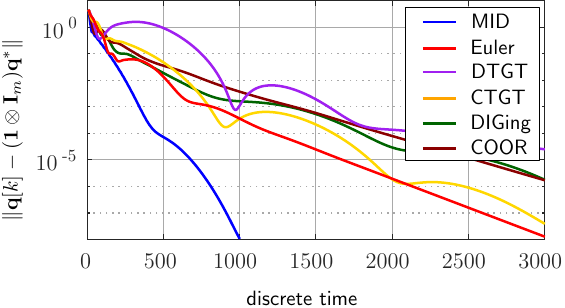}
\caption{Comparison of MID, Euler discretization as well as DTGT \cite{Notarstefano2019}, CTGT \cite{guido2023}, DIGing \cite{nedic2017achieving} and COOR \cite{kia2015distributed} for a quadratic optimization problem. The fastest values of $\tau$ were found to be approximately $3.78, 0.71, 0.58, 0.65, 0.05, 0.078$. The proposal MID has a clear advantage in terms of speed of convergence.}
\label{fig:comparison:smalltau}
\end{figure}

\begin{figure}
\centering
\includegraphics[width=0.5\textwidth]{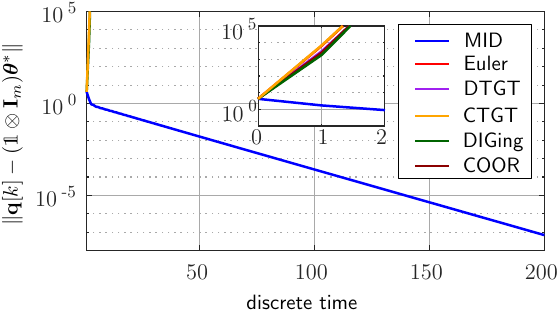}
\caption{Comparison of MID, Euler discretization as well as DTGT \cite{Notarstefano2019}, CTGT \cite{guido2023}, DIGing \cite{nedic2017achieving} and COOR \cite{kia2015distributed} for a quadratic optimization problem with $\tau=10$.}
\label{fig:comparison:bigtau}
\end{figure}

\begin{figure}
\centering
\includegraphics[width=0.5\textwidth]{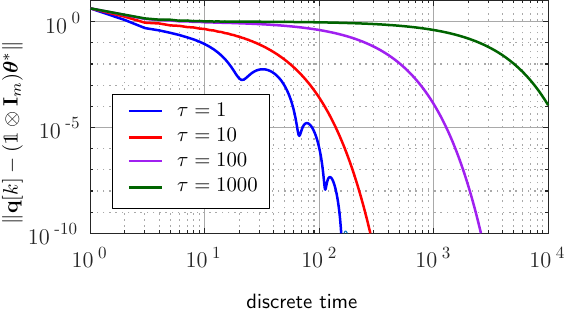}
\caption{Comparison of MID for a quadratic optimization problem with $\tau=1, 10, 100, 1000$.}
\label{fig:multipletau}
\end{figure}
\section{Conclusions}
\label{sec:conclusions}
In this work, we introduced a novel distributed optimization approach that achieves parameter-free convergence, a significant departure from traditional methods reliant on precise parameter selection, such as the learning rate. Using the Mixed Implicit Discretization (MID) of a continuous-time port-Hamiltonian system, we demonstrated how this system-based strategy enhances convergence speeds without worrying about the relationship between parameters and stability. Our method leveraged the inherent structure of the port-Hamiltonian system to ensure convergence by aligning the system's equilibrium with the optimization problem's solution, even when transitioning from continuous to discrete time. Through the MID scheme, we enabled local resolution of nonlinear implicit equations by network agents, formally proving the algorithm's parameter-free convergence capability. Extensive numerical analyses highlighted our approach's superiority in scenarios where traditional methods fail due to parameter constraints, thereby significantly accelerating convergence speeds.

\appendices

\section{Convex functions}
\label{ap:convex}
This section provides some useful facts used in the rest of the manuscript. We refer the reader to \cite{boyd} for standard definitions of convex functions. 
\begin{proposition}
\label{prop:boyd}\cite[Section 9.1.2]{boyd}
Let $g:\mathbb{R}^n\to\mathbb{R}$ be a strongly convex differentiable function. Then,
    $$
    g(\mf{v})\geq g(\mf{u})+\nabla g(\mf{u})^\top(\mf{v}-\mf{u}) + \frac{\mu}{2}\|\mf{v}-\mf{u}\|^2
    $$
    for arbitrary $\mf{u},\mf{v}\in\mathbb{R}^n$ and some $\mu>0$ referred to as the strong convexity constant. As a consequence,
    $$
    (\mf{v}-\mf{u})^\top(\nabla g(\mf{v})-\nabla g(\mf{u}))\geq \mu\|\mf{v}-\mf{u}\|^2
    $$
\end{proposition}

\begin{lemma}
\label{le:lyapunov}
Let $H:\mathbb{R}^n\to\mathbb{R}$ be a differentiable function. Moreover, let some fixed $\mf{x}^*\in\mathbb{R}^n$ and $V_{\mf{x}^*}:\mathbb{R}^n\to\mathbb{R}$ be defined as $$
    V_{\mf{x}^*}(\mf{x}) = H(\mf{x}) - H(\mf{x}^*) - \nabla H(\mf{x}^*)^\top(\mf{x}-\mf{x}^*).$$
Then, If $H$ is strongly convex with constant $\mu>0$, $V_{\mf{x}^*}(\mf{x})$ satisfy the radially unboundedness condition given by:
$$
V_{\mf{x}^*}(\mf{x})\geq \frac{\mu}{2}\|\mf{x}-\mf{x}^*\|^2
$$
\end{lemma}
\begin{proof}
    Use Proposition \ref{prop:boyd} with $g = H$, $\mf{v}=\mf{x}$ and $\mf{u}=\mf{x}^*$:
    $$
    \begin{aligned}
    &V_{\mf{x}^*}(\mf{x}) =H(\mf{x}) - H(\mf{x}^*)-\nabla H(\mf{x}^*)^\top(\mf{x}-\mf{x}^*) \\
    &\geq \left(H(\mf{x}^*)+\nabla H(\mf{x}^*)^\top(\mf{x}-\mf{x}^*) + \frac{\mu}{2}\|\mf{x}-\mf{x}^*\|^2\right)\\&- H(\mf{x}^*)-\nabla H(\mf{x}^*)^\top(\mf{x}-\mf{x}^*)=\frac{\mu}{2}\|\mf{x}-\mf{x}^*\|^2
    \end{aligned}
    $$
\end{proof}

\section{Proof of Theorem \ref{th:main}}
\label{ap:proof}
Through this section, for any variable $\mf{x}$, the notations $\overline{\bullet}, \tilde{\bullet}$ act as $\overline{\mf{x}}=(\mf{x}^++\mf{x})/2, \tilde{\mf{x}}=\mf{x}^+-\mf{x}$. We start by showing the well-posedness of \eqref{eq:main:algorithm}.
\begin{lemma}
\label{lem:well:posedness}
    Consider the conditions of Theorem \ref{th:main}. Then, for any $\{\mf{x}_i\}_{i\in\mathcal{I}}$ there always exist a unique $\{\mf{x}_i^+\}_{i\in\mathcal{I}}$ satisfying \eqref{eq:main:algorithm}.
\end{lemma}
\begin{proof}
    First, recall the partition $\mf{x}_i=[\mf{q}_i^\top,\mf{p}_i^\top]^\top$. In this context, and with the choices of $\mf{M}, \bm{\phi}_i$ one can solve for $\mf{p}_i^+$ as
    \begin{equation}
    \label{eq:pi}
    \mf{p}_i^+  = \mf{p}_i+\tau\left(|\mathcal{N}_i|\mf{q}_i^+-\sum_{j\in\mathcal{N}_i}\mf{q}_j\right) 
    \end{equation}
   Substitute \eqref{eq:pi} into the update equation for $\mf{q}_i$ in \eqref{eq:main:algorithm}. Rearranging:
    \begin{equation}
    \label{eq:condition:existence}
    \mf{G}\mf{q}_i^++\nabla f_i\left(\frac{\mf{q}_i^++\mf{q}_i}{2}\right)+\mf{c}=0
    \end{equation}
with 
$$
\begin{aligned}
\mf{G}&=(1/\tau+|\mathcal{N}_i|+\tau|\mathcal{N}_i|^2)\mf{I}_m\\ \mf{c} &= -\sum_{j\in\mathcal{N}_i}\left((1+\tau)\mf{q}_j+\mf{p}_j\right)+\mf{q}_i/\tau
\end{aligned}
$$
Note that \eqref{eq:condition:existence} is equivalent to $\nabla g(\mf{q}_i^+)=\mf{0}$ where $\nabla$ is taken with respect to $\mf{q}_i^+$ and with
$$
g(\mf{q}_i^+) = \frac{1}{2}(\mf{q}_i^+)^\top\mf{G}\mf{q}_i^+ + 2f_i\left(\frac{\mf{q}_i^++\mf{q}_i}{2}\right)+\mf{c}^\top\mf{q}_i^+
$$

Note that $\mf{G}\succ \mf{0}$ and $f_i$ is a convex function. Hence, $g$ is convex as well. Henceforth, there exist a unique solution to $\nabla g(\mf{q}_i^+)=\mf{0}$, which coincides with the solution of the convex optimization program $\min_{\mf{q}_i^+\in\mathbb{R}^m}g(\mf{q}_i^+)$. This implies the existence of a unique $\mf{p}_i^+$ as well from \eqref{eq:pi}.
\end{proof}
Now, we write \eqref{eq:main:algorithm} in a more convenient representation.

\begin{lemma}
\label{lem:system}
    Consider the conditions of Theorem \ref{th:main} and let ${\mf{q}}=[\mf{q}_i]_{i\in\mathcal{I}}$ and ${\mf{p}}=[\mf{p}_i]_{i\in\mathcal{I}}$. Moreover, let $\mf{A},\mf{D}=\text{\normalfont diag}(\mf{A}\mathds{1})$ denote the adjacency and degree matrices of $\mathcal{G}$ and define
    \begin{equation}
    \label{eq:matrices}
    \begin{aligned}
        &{\mf{L}_m} = (\mf{D}-\mf{A})\otimes\mf{I}_m,\quad {\mf{Q}} = (\mf{D}+\mf{A})\otimes\mf{I}_m/2 \\
        &\bm{\Lambda}(\tau)=\frac{1}{\tau^2}\mf{I}+\frac{1}{\tau}{\mf{Q}}+{\mf{Q}}^2,\\
        &    \mf{S}_p(\tau)=\begin{bmatrix}
        -\bm{\Lambda}(\tau)^{-1}(\frac{1}{\tau}\mf{I}+{\mf{Q}})\mf{L}_m & -\frac{1}{\tau}\bm{\Lambda}(\tau)^{-1}\mf{L}_m \\
        \frac{1}{\tau}\bm{\Lambda}(\tau)^{-1}\mf{L}_m & -\bm{\Lambda}(\tau)^{-1}{\mf{Q}}\mf{L}_m 
        \end{bmatrix}\\
        &    \mf{S}_r(\tau) = \begin{bmatrix}
    -\bm{\Lambda}(\tau)^{-1}(\frac{1}{\tau}\mf{L}_m + {\mf{Q}}\mf{L}_m+\mf{L}_m{\mf{Q}})& -\frac{1}{\tau}\bm{\Lambda}(\tau)^{-1}\mf{L}_m\\
    \mf{L}_m & \mf{0}
    \end{bmatrix}
    \end{aligned}
    \end{equation}
    
    Then, $\bm{\Lambda}(\tau)$ is invertible. Moreover, using $\mf{r}=\mf{p}-{\mf{Q}}\mf{q}$, trajectories of \eqref{eq:main:algorithm} satisfy:
        \begin{equation}
        \label{eq:mid:sys2}
        \begin{bmatrix}
            \tilde{\mf{q}} \\
            \tilde{\mf{r}}
        \end{bmatrix} = \mf{S}_r(\tau)\begin{bmatrix}
            \overline{\mf{q}} \\
            \overline{\mf{r}}
        \end{bmatrix}-\begin{bmatrix}
            \bm{\Lambda}(\tau)^{-1}/\tau\\
           \mf{0}
        \end{bmatrix}\nabla f(\overline{\mf{q}})
    \end{equation}
\end{lemma}
\begin{proof}
    First, write \eqref{eq:discrete:MAS:MID} as:
    \begin{equation}
    \label{eq:compact:discrete1}
    \tilde{\mf{x}}/\tau = (\mf{D}\otimes\mf{M})\mf{x}^+-(\mf{A}\otimes\mf{M})\mf{x}+\bm{\phi}(\overline{\mf{x}})
    \end{equation}
Now, note that the following relations hold:
    $$
    \begin{aligned}
    \mf{x}^+ &= \overline{\mf{x}}+\tilde{\mf{x}}/2\\
    \mf{x} &= \overline{\mf{x}}-\tilde{\mf{x}}/2
    \end{aligned}
    $$
    replacing these expressions in \eqref{eq:compact:discrete1} and rearranging results in:
    $$
    (\mf{I}/\tau -(\mf{A}+\mf{D})\otimes\mf{M}/2 )\tilde{\mf{x}} = (\mf{L}\otimes\mf{M})\overline{\mf{x}}+\bm{\phi}(\overline{\mf{x}})
    $$
    In terms of the partitions of $\mf{x}$ and the expression for $\mf{M}$ as in \eqref{eq:M}:
    \begin{equation}
    \label{eq:compact2}
    \begin{aligned}
    \begin{bmatrix}
    \mf{I}/\tau+{\mf{Q}}&{\mf{Q}}\\
    -{\mf{Q}}&\mf{I}/\tau
    \end{bmatrix}\begin{bmatrix}
    \tilde{\mf{q}}\\
    \tilde{\mf{p}}
    \end{bmatrix}=
    \begin{bmatrix}
    -\mf{L}_m&-\mf{L}_m\\
    \mf{L}_m&\mf{0}
    \end{bmatrix}\begin{bmatrix}
    \overline{\mf{q}}\\
    \overline{\mf{p}}
    \end{bmatrix} - \begin{bmatrix}
    \mf{I} \\
    \mf{0}
    \end{bmatrix}\nabla f(\overline{\mf{q}})
    \end{aligned}
    \end{equation}
Note that:
$$
\begin{bmatrix}
    \mf{I}/\tau+{\mf{Q}}&{\mf{Q}}\\
    -{\mf{Q}}&\mf{I}/\tau
    \end{bmatrix}^{-1} = \bm{\Lambda}(\tau)^{-1}\begin{bmatrix}
\mf{I}/\tau & -{\mf{Q}} \\
{\mf{Q}} & \mf{I}/\tau + {\mf{Q}}
    \end{bmatrix}
$$
Invertibility of $\bm{\Lambda}(\tau)$ comes from the fact that ${\mf{Q}}=\mf{D}+\mf{A}$ is diagonally dominant and thus is positive semi-definite \cite[Section 6.1]{horn}. Hence, ${\mf{Q}}^2$ is positive semi-definite, being $\bm{\Lambda}(\tau)$ positive definite and thus invertible. Therefore, from \eqref{eq:compact2} it follows that:
    \begin{equation}
        \label{eq:mid:sys1}
        \begin{aligned}
        &\begin{bmatrix}
            \tilde{\mf{q}} \\
            \tilde{\mf{p}}
        \end{bmatrix} = \mf{S}_p(\tau)\begin{bmatrix}
            \overline{\mf{q}} \\
            \overline{\mf{p}}
        \end{bmatrix}-\begin{bmatrix}
            \bm{\Lambda}(\tau)^{-1}/\tau\\
            \bm{\Lambda}(\tau)^{-1}{\mf{Q}}
        \end{bmatrix}\nabla f(\overline{\mf{q}})\\
        &    \mf{S}_p(\tau)=\begin{bmatrix}
        -\bm{\Lambda}(\tau)^{-1}(\frac{1}{\tau}\mf{I}+{\mf{Q}})\mf{L}_m & -\frac{1}{\tau}\bm{\Lambda}(\tau)^{-1}\mf{L}_m \\
        \frac{1}{\tau}\bm{\Lambda}(\tau)^{-1}\mf{L}_m & -\bm{\Lambda}(\tau)^{-1}{\mf{Q}}\mf{L}_m 
        \end{bmatrix}\\
        \end{aligned}
    \end{equation}
The rest of the lemma follows from the change of variables
$$
\begin{aligned}
&\begin{bmatrix}
{\mf{q}} \\
{\mf{r}}
\end{bmatrix} = \mf{T}\begin{bmatrix}
{\mf{q}} \\
{\mf{p}}
\end{bmatrix}
\end{aligned}
$$
with 
$$
\begin{aligned}
&\quad\mf{T} = \begin{bmatrix}
\mf{I} & \mf{0} \\
-{\mf{Q}}/\tau & \mf{I}
\end{bmatrix}, \quad\mf{T}^{-1} = \begin{bmatrix}
\mf{I} & \mf{0} \\
{\mf{Q}}/\tau  & \mf{I}
\end{bmatrix}
\end{aligned}
$$
Note that $\bm{\Lambda}(\tau)$ and ${\mf{Q}}$ commute so that ${\mf{Q}}$ also commutes with $\bm{\Lambda}(\tau)^{-1}$. From this, a straightforward calculation can verify that:
$$
\begin{aligned}
&\mf{T}\begin{bmatrix}
            \bm{\Lambda}(\tau)^{-1}/\tau\\
            \bm{\Lambda}(\tau)^{-1}{\mf{Q}}
        \end{bmatrix} = \begin{bmatrix}
            \bm{\Lambda}(\tau)^{-1}/\tau\\
            \mf{0}
        \end{bmatrix}\\ &\mf{S}_r(\tau) = \mf{T}\mf{S}_p(\tau)\mf{T}^{-1}
\end{aligned}
$$
from which \eqref{eq:mid:sys2} is obtained, completing the proof.
\end{proof}
Given these results, we can show Theorem \ref{th:main}.

\begin{proof}(Of Theorem \ref{th:main})
    First, use Lemma \ref{lem:system} to write \eqref{eq:main:algorithm} as \eqref{eq:mid:sys2}. Denote with $\mf{y}=[\mf{q}^\top,\mf{r}^\top]^\top$. Let $\mf{y}^*$ the equilibrium of \eqref{eq:mid:sys2} and set $\mf{e}=\mf{y}-\mf{y}^*$. Hence, it follows from \eqref{eq:mid:sys2} that:
    $$
    \begin{aligned}
    \tilde{\mf{e}} &= \mf{S}_r(\tau)\overline{\mf{y}}-\mf{B}(\tau)\nabla f(\overline{\mf{q}})\\
    &=\mf{S}_r(\tau)\overline{\mf{e}}-\mf{B}(\tau){\bm{\psi}}(\overline{\mf{q}})
    \end{aligned}
    $$
    with $\mf{B}(\tau) = [ \Lambda(\tau)^{-1}/\tau, \mf{0}]^\top$ and $\bm{\psi}(\overline{\mf{q}})=\nabla f(\overline{\mf{q}})-\nabla f({\mf{q^*}})$.

    Now, assume that \eqref{eq:lmi} holds and set a Lyapunov function candidate $$V(\mf{e})=\mf{e}^\top\mf{P}(\tau)\mf{e}/2$$ with $\mf{P}(\tau)$ in \eqref{eq:lmi:P}. Hence,
    $$
    \begin{aligned}
    &V(\mf{e}^+)-V(\mf{e}) = \overline{\mf{e}}^\top\mf{P}(\tau)\tilde{\mf{e}}\\
    &=\overline{\mf{e}}^\top(\mf{P}(\tau)\mf{S}_r(\tau)+\mf{S}_r(\tau)^\top\mf{P}(\tau))\overline{\mf{e}}-\overline{\mf{e}}^\top\mf{P}(\tau)\mf{B}(\tau){\bm{\psi}}(\overline{\mf{q}})
    \end{aligned}
    $$
    Now, note that $\mf{P}(\tau)\mf{B}(\tau)=[\mf{I}/\tau, \mf{P}_{12}\bm{\Lambda}(\tau)/\tau]^\top$ such that
    \begin{equation}
    \label{eq:eB}
    \begin{aligned}
&-\overline{\mf{e}}^\top\mf{P}(\tau)\mf{B}(\tau){\bm{\psi}}(\overline{\mf{q}}) = -(\overline{\mf{q}}-\mf{q}^*)(\nabla f(\overline{\mf{q}})-\nabla f({\mf{q^*}}))/\tau\\
&-(\overline{\mf{r}}-\mf{r}^*)^\top\mf{P}_{12}^\top\bm{\Lambda}(\tau)\bm{\psi}(\overline{\mf{q}})/\tau\\
&\leq -(\mu\mfs{N}/\tau)\|\overline{\mf{q}}-\mf{q}^*\|^2+\gamma(\tau)\|\mf{P}_{12}(\overline{\mf{r}}-\mf{r}^*)\|\|\overline{\mf{q}}-\mf{q}^*\|
    \end{aligned}
    \end{equation}
where we used Proposition \ref{prop:boyd} due to strong convexity of $f$,  used $$\|(\nabla f(\overline{\mf{q}})-\nabla f({\mf{q^*}})\|\leq L \|\overline{\mf{q}}-\mf{q}^*\|$$ due to the Lipschitz property of $\nabla f$ and setting  
$$
\gamma(\tau)=(L/\tau)\|\bm{\Lambda}(\tau)\|
$$ Moreover, one can use Young's inequality to get
$$
\begin{aligned}
&-(\mu\mfs{N}/\tau)\|\overline{\mf{q}}-\mf{q}^*\|^2+\gamma(\tau)\|\mf{P}_{12}(\overline{\mf{r}}-\mf{r}^*)\|\|\overline{\mf{q}}-\mf{q}^*\|\\
    &\leq \left(-\frac{\mu\mfs{N}}{\tau}+\frac{\varepsilon\gamma(\tau)}{2}\right)\|\overline{\mf{q}}-\mf{q}^*\|^2+\frac{\gamma(\tau)}{2\varepsilon}\|\mf{P}_{12}(\overline{\mf{r}}-\mf{r}^*)\|^2\\
&\leq \overline{\mf{e}}^\top \mf{S}_\gamma(\tau,\varepsilon)\overline{\mf{e}}
\end{aligned}
$$
with
\begin{equation}
\label{eq:S:gamma}
\mf{S}_\gamma(\tau,\varepsilon) = \begin{bmatrix}
(\gamma(\tau)\varepsilon/2-\mu\mfs{N}/\tau)\mf{I} & \mf{0} \\
\mf{0} & \gamma(\tau)\mf{U}/(2\varepsilon)
\end{bmatrix}
\end{equation}
where $\mf{P}_{12}^\top\mf{P}_{12}\preceq \mf{U}$, which follows from \eqref{eq:lmi:U} due to a Schur complement.
Henceforth:
\begin{equation}
\label{eq:V:decrease}
\begin{aligned}
V(\mf{e}^+)-V(\mf{e})&\leq \overline{\mf{e}}^\top(\mf{P}(\tau)\mf{S}_r(\tau)+\mf{S}_r(\tau)^\top\mf{P}(\tau)+\mf{S}_\gamma(\tau))\overline{\mf{e}}\\
&\leq -u\|\overline{\mf{q}}-\mf{q}^*\|^2
\end{aligned}
\end{equation}
due to \eqref{eq:lmi:S}. Hence, $V(\mf{e})$ decrease unless $\overline{\mf{q}}=\mf{q}^*$. However, $\overline{\mf{q}}=\mf{q}^*$ implies $\overline{\mf{q}}=(\mathds{1}\otimes\mf{I}_m)\bm{\theta}^*$ (see Example \ref{ex:distro}) such that $\tilde{\mf{r}}=\mf{L}_m\overline{\mf{q}}=\mf{0}$ from \eqref{eq:mid:sys2}, meaning that equilibrium was reached for $\tilde{\mf{r}}$ as well. Henceforth, $\mf{e}[k]$ converges asymptotically to the origin, completing the proof.
\end{proof}

\section{Proof of the corollaries}
\label{ap:cor}
\subsection{Proof of Corollary \ref{cor:special}}
Set $\mf{P}_{12}=\mf{U}=\mf{0}$, $\mf{P}_{22}=\mf{I}/\tau$ and $u=\mu, \varepsilon=0$. With these choices, \eqref{eq:lmi:P} is complied since $\bm{\Lambda}(\tau)$ is positive definite and \eqref{eq:lmi:U} is complied trivially. Moreover, $\mf{W}(\tau)=\mf{P}(\tau)\mf{S}_r(\tau)+\mf{S}_r(\tau)^\top\mf{P}(\tau)+\mf{S}_\gamma(\tau, 0)$ for these choices which result in
$$
\mf{W}(\tau)=\begin{bmatrix}
    -(\frac{1}{\tau}\mf{L}_m + {\mf{Q}}\mf{L}_m+\mf{L}_m{\mf{Q}}) - \mu\mf{I} &\mf{0}\\
    \mf{0} & \mf{0}
\end{bmatrix}
$$
Since $\mf{L}_m$ is positive semi-definite and $$\mf{Q}\mf{L}_m+\mf{L}_m\mf{Q} = (\mf{D}^2-\mf{A}^2)\otimes \mf{I}_m\succeq\mf{0}$$ due to \eqref{eq:special}, therefore \eqref{eq:lmi:S} is complied for these choices for all $\tau>0$. Note that if $\mathcal{G}$ is a cycle, then $\mf{D}=2\mf{I}$ and $\mf{D}^2-\mf{A}^2 = 4\mf{I}-(2\mf{I}+\mf{A}_2)$ where $\mf{A}_2$ is the adjacency matrix of the graph $\mathcal{G}_2$ with edges connecting agents two steps away in the original graph $\mathcal{G}$. This means that $\mf{A}_2$ is full of zeros, except for two 1's per row. Henceforth, $$\mf{D}^2-\mf{A}^2=2\mf{I}_\mfs{N}-\mf{A}_2$$ is diagonally dominant and thus positive semi-definite \cite[Section 6.1]{horn}. Similarly, if $\mathcal{G}$ is fully connected, then $\mf{D}^2-\mf{A}^2 = \mfs{N}\mf{I}_\mfs{N}-\mfs{A}^2$ and $\mf{A}=\mathds{1}\mathds{1}^\top-\mf{I}_\mfs{N}$ so that $$\mf{D}^2-\mf{A}^2=(\mfs{N}-1)\mf{I}_\mfs{N}-(\mfs{N}-2)\mathds{1}\mathds{1}^\top.$$ Similarly as before, $\mf{D}^2-\mf{A}^2$ is diagonally dominant and thus positive semi-definite.

\subsection{Proof of Corollary \ref{cor:bound}}
Make the same choices as in the proof of Corollary \ref{cor:special}, but let $\mathcal{G}$ be an arbitrary connected undirected graph and arbitrary $u>0$. Henceforth, 
$$
\begin{aligned}
&-(\frac{1}{\tau}\mf{L}_m + {\mf{Q}}\mf{L}_m+\mf{L}_m{\mf{Q}}) - \frac{\mu}{\tau}\mf{I} \\&\preceq \left(\|{\mf{Q}}\mf{L}_m+\mf{L}_m{\mf{Q}}\|-\frac{\mu}{\tau}\right)\mf{I}\preceq -u\mf{I}\end{aligned}
$$
by the bound on $\tau$ for some $u>0$. Hence, \eqref{eq:lmi:S} is complied due to ${\mf{Q}}\mf{L}_m+\mf{L}_m{\mf{Q}}=(\mf{D}^2-\mf{A}^2)\otimes \mf{I}_m$.

\subsection{Proof of Corollary \ref{cor:linear}}
First note that
$$
\bm{\psi}(\overline{\mf{q}})=\nabla f(\overline{\mf{q}})-\nabla f({\mf{q^*}}) = \mf{H}(\overline{\mf{q}}-\mf{q}^*)
$$
with $\mf{H}=\text{\normalfont blockdiag}(\mf{H}_1,\dots,\mf{H}_\mfs{N})$. Now, we follow all steps from the proof of Theorem \ref{th:main} but instead of \eqref{eq:eB}, do:
 \begin{equation}
    \begin{aligned}
&-\overline{\mf{e}}^\top\mf{P}(\tau)\mf{B}(\tau){\bm{\psi}}(\overline{\mf{q}}) = \overline{\mf{e}}^\top\mf{S}_h(\tau)\overline{\mf{e}}
    \end{aligned}
    \end{equation}
with
\begin{equation}
\label{eq:Sh}
\mf{S}_h(\tau) =\begin{bmatrix}
\mf{H}/\tau & \mf{0} \\
\mf{P}_{12}^\top\bm{\Lambda}(\tau)\mf{H}/\tau & \mf{0}
\end{bmatrix}
\end{equation}
Therefore, similar to \eqref{eq:V:decrease}:

\begin{equation}
\begin{aligned}
V(\mf{e}^+)-V(\mf{e})&\leq \overline{\mf{e}}^\top(\mf{P}(\tau)\mf{S}_r(\tau)+\mf{S}_r(\tau)^\top\mf{P}(\tau)+\mf{S}_h(\tau))\overline{\mf{e}}\\
&\leq -u\|\overline{\mf{q}}-\mf{q}^*\|^2
\end{aligned}
\end{equation}
from which convergence follows using the same arguments at the end of the proof of Theorem \ref{th:main}.

\bibliographystyle{IEEEtran}

\end{document}